\documentclass[12pt,a4paper]{amsart}
\usepackage{url}

\usepackage[utf8]{inputenc}
\usepackage[T1]{fontenc}
\usepackage{amsmath}
\usepackage{amsfonts}
\usepackage{amssymb}
\usepackage{amscd}
\usepackage{natbib}

\usepackage[bookmarksopen=false,breaklinks=true,%
      backref=page,pagebackref=true,plainpages=false,%
      hyperindex=true,pdfstartview=FitH,%
      pdfpagelabels=true,
      colorlinks=true,linkcolor=blue,%
      citecolor=red,urlcolor=red,
      colorlinks=true%
      ]%
   {hyperref}

\let\its\itshape

\let\ss\smallskip
\let\ms\medskip

\let\noi\noindent
\newcommand\mni{\ms\noi}

\newcommand\ie{i.e., }
\newcommand\eg{e.g., }

\let\sc\scriptstyle

\newcommand\B{\mathbb{B}}
\newcommand\N{\mathbb{N}}
\newcommand\Z{\mathbb{Z}}
\newcommand\Q{\mathbb{Q}}
\newcommand\R{\mathbb{R}}
\newcommand\F{\mathbb{F}}
\renewcommand\k{\mathbf{k}}
\newcommand\gl{\mathbf{l}}
\newcommand\A{\mathbb{A}}
\renewcommand\Im{\mathrm{Im}}
\def\I{\mathrm{Id}}
\newcommand\Rad{\mathrm{Rad}}

\def\m{\mathfrak{m}}
\def\fp{\mathfrak{p}}
\def\fa{\mathfrak{a}}

\def\S{\mathfrak{S}}

\newcommand\Dk{D_\k}
\newcommand\X{\mathbf{X}}
\newcommand\x{\mathbf{x}}

\def\tq{\>:\>}

\def\ou{\mathrel{\vee}}

\let\dans\longrightarrow

\newcommand\ov{\overline}
\renewcommand\l{\left}
\renewcommand\r{\right}
\newcommand\la{\l\langle}
\newcommand\ra{\r\rangle}
\newcommand\flr[1]{\l/\la #1\ra\r.}

\let\epsilon\varepsilon

\newtheorem{theorem}{Theorem}[section]
\newtheorem{proposition}[theorem]{Proposition}
\newtheorem{defi}[theorem]{Definition}
\newtheorem{example}[theorem]{Example}
\newtheorem{lemma}[theorem]{Lemma}
\newtheorem{coro}[theorem]{Corollary}

\newtheorem{context}{Context}[section]

\begin{document}
\title[Elementary Constructive Theory of Henselian Local Rings]
{Elementary Constructive Theory\\ of Henselian Local Rings}
\author{M. Emilia Alonso Garcia}
\thanks{Universitad Complutense, Madrid, Espa\~na. {\tt mariemi@mat.ucm.es}}
\author{Henri Lombardi}
\thanks{Univ. Marie et Louis Pasteur, 25030
Besan\c{c}on cedex, France. {\tt henri.lombardi@univ-fcomte.fr}
\url{http://hlombardi.free.fr/}}
\author{Hervé Perdry}
\thanks{Université Paris-Saclay, UVSQ, Inserm, CESP, 16 av PV Couturier, 94807,
Villejuif, France. {\tt herve.perdry@universite-paris-saclay.fr} }
\date{May 2006, revised August 2025}
\maketitle

This is the arXiv version, arXiv:2202.06595, written in August 2025. 

This paper appeared as 

Mar\'{\i}a Emilia {Alonso Garc\'{\i}a}, Henri {Lombardi} and  Hervé {Perdry}. {Elementary constructive theory of {H}enselian local rings}.
p.\ 253--271 in
 {\it MLQ. Mathematical Logic Quarterly},
    Vol {54}, No~3. (2008).
    
Here we have fixed the reference for the book \cite{LafonMarot} and for the Cayley-Hamilton Theorem in Lemma \ref{charpol}, and we have added the reference \cite{CACM}.

In this version, we use the terminology found in the book \cite{CACM} which appeared after the publication of the paper.\\ 
$\bullet$ The Jacobson radical of the ring $A$ is denoted as $\mathrm{Rad}(A)$.\\
$\bullet$ A ``pseudo-local ring'' is now called a ``residually zero-dimensional ring''.

Note that ``semi-local ring'' here is a particular case of ``strict semi-local ring'' described in exercices IX-8 and IX-9 of \cite{CACM}.

Moreover, in order to make easier to understand the structure of the paper, we have numbered the various {\bf contexts}, and we have indicated the corresponding context for important results.

A table of contents is given at the end of the paper.      


\medskip \centerline{-----------------------------------}  

\begin{abstract}
    We give an elementary theory of Henselian local rings and
    construct the Henselization of a local ring.  All our theorems
    have an algorithmic content.
\end{abstract}

\medskip\noindent Key words: Henselian local ring, Henselization, strict Henselian local ring, 
Constructive mathematics

\medskip\noindent MSC 2000: 3J15, 03F55


\section{Introduction}

We give an elementary theory of Henselian local rings.  The paper is
written in the style of Bishop's constructive mathematics,
\ie mathematics with intuitionistic logic (see \cite{BB85,BR1987,MRR}. In this frame  we do not assume any constraint of the kind 
``explicit means Turing computable''. So that, our proofs work as well inside classical mathematics; it is sufficient to assume that ``explicit'' is a void word. However if the hypotheses are ``Turing computable'', so are the conclusions. In this sense we claim  that our proofs  have always an algorithmic content.

Through this paper, when we say: ``Let \(R\) be a ring \ldots '', this means that:\\
(1)~we know how 
to construct elements of \(R\) (from now on  called   {\it canonical elements}), (2)~we have given  \(1_R\), \(0_R\),  \(-1_R\), constructed according to (1),  (3)~we know how 
to construct \(x+y\) and \(xy\) according to (1),  when the objects $x$ and $y$ are given through the same construction, (4)~we know what is the meaning of
\(x=_{R}y\) when \(x\) and \(y\) are 
elements of \(R\) given through the construction (1), and (5)~we have constructive proofs
showing that the axioms of rings are satisfied by this structure.

Hence, \(\Z\), \(\Q\), \(\R\), and all usual rings are rings in the constructive meaning of the word already explained. Notice that (4) does not imply to have given a constructive test of equality for canonical elements. 
A ring \(R\) is said \emph{discrete} when we have, constructively, 
for any  elements \(x\) and \(y\)  of \(E\): \(x=_{E}y\) or \(\lnot(x=_{E}y)\).
So \(\R\) is \emph{not} discrete. If it were the
case, this would imply the following so-called {\em limited principle of
omniscience}:
\[ \mathbf{(LPO)}\ \ \forall \alpha\in\{0,1\}^\N,\ 
  (\exists n,\ \alpha_n=1)\ou (\forall n,\ \alpha_n=0) \]
which is considered to be not acceptable in constructive mathematics.
For more details, we refer the reader to \cite{BR1987,MRR}.

Of course, in classical mathematics, all 
constructive theorems
about discrete fields apply to~$\R$ because
it becomes discrete if we assume $\mathbf{(LPO)}$.

On the other hand \(\Z\) is a discrete ring, even if 
for ``non-canonical''elements of \(\Z\) it is impossible to decide
the equality (\eg equality between \(x\) and \(y\) where \(x=0\) and \(y=0\) if ZF is consistent, and \(y=1\) in the 
other case).

Many classical definitions have to be rewritten in a suitable
form  to fit well in our constructive setting. 
E.g., a \emph{local ring}, will be a ring \(A\) such that
\[\forall x\in A,\ x\in A^\times\hbox{\ or\ }(1+x)\in A^\times.\]
Precisely, this will means that, for any canonical \(x\in
A\) we can either construct  \(y\) such that \(xy=1\), or
construct  \(y\) such that \((1+x)y=1\), with an explicit meaning for the ``or''.

This construction is not required to be ``extensional'': two 
(canonical) 
elements \(x\) and \(x'\) of \(A\) which are equal in \(A\), need not give the 
same branch of the alternative. Typically, \(\R\) is a local ring in 
which there \emph{cannot} exist an extensional way of satisfying the axiom of local rings.

\medskip The theory of Henselian rings was developed by Nagata
(see \cite{Nagata62}). He introduced the notion  of  Henselization of a local ring based on the case of integrally closed domains.
Namely, for an  integrally closed domain $(R,\m)$, let $R'$ the separable integral closure  of $R$ inside an algebraic closure of its quotient field. Let  $\fp'$ a maximal ideal  of $R'$ lying over $\m$, $R''$ be the splitting ring of $\fp'$ and
$\fp''=R''\cap \fp'$. The Henselization of $R$ is $R''_{\fp''}$
This ``construction'' is very abstract, and it seems difficult to be adapted to our constructive setting.

Our approach relies heavily upon more recent expositions, namely the book \cite{LafonMarot},
Chapters 12 \& 13. Although this book is not written in a purely constructive way, the authors have made a remarkable effort in order to give simplified proofs of many classical results, so that, it provided us a good basis to develop our constructive theory.
In fact, in \cite{LafonMarot} the Henselization of a local ring is constructed as a part of an inductive completion of the ring \(A\). Namely,  one consider  the inductive limit of the family of sub-rings obtained by taking the
completion of local Noetherian sub-rings of \(A\). This is a natural object, but it is difficult to manage it from constructive point of
view. Our achievement could be considered as simplifying or making explicit the
\cite{LafonMarot} construction.

There are still some problems to be solved in order to have  a satisfactory constructive theory of Henselian local rings 
In particular, to prove the so-called {\its multi-dimensional Hensel Lemma}
whose proof relies on the Zariski Main Theorem, which
is highly non constructive.
In classical mathematics, the
Henselization of a local ring \((A,\m)\) coincides with the limit of the local essentially
finitely generated \(A\)-algebras  
\( \left(
          A[X_1,\ldots, X_n]\flr{F_1, \ldots, F_n}
   \right)_{ (\m,x)}\) 
at a non
singular point \((\m,x)\). This fact allows to represent algebraic functions
(locally), and to state algorithms on standard bases in the ring of algebraic formal power series (cf.\ \cite{AMR}). This
characterization of the Henselization relies on  Zariski Main Theorem, which provides a kind 
of ``primitive smooth element'' for étale extensions.

Finally, it will be also interesting to compare explicitly
our construction to the one, already quoted, given in 
\cite{Nagata62}.

The plan of the paper is as follows. The first three sections are devoted to study the basic notions that will be useful for our constructive proofs; namely the Boolean algebra  of idempotents  of a finite $A$ algebra, and the Universal decomposition algebra. Idempotents  play a similar role to minimal prime ideals, that are in general non constructable objects. 

In 4 \& 5 we accomplish the construction of the Henselization of a local ring, and 
we prove some basic properties of
Henselian local rings, including the fact that residual idempotents in
a finite algebra over a Henselian ring  can be always  lifted to idempotents of the
algebra. Finally  we generalize our approach  to the construction of the strict Henselization.

We recall that  the constructive theory of Henselian valuation rings has been 
developed in \cite{KL00}, \cite{Per01} and \cite{Per05}.


\section{Rings and Local Rings}

In the whole paper, rings are commutative.

\subsection{Radicals}

 The {\em Jacobson radical} of a ring \(A\) is
\[ \Rad(A) = \{ x\in A\tq \forall y\in A,\ 1+x y \in A^\times \}. \]

Let \(A\) be a ring and \(I\subseteq A\) an ideal. The {\em
nilradical} of \(I\) is
\[\sqrt I=\{ x \in A \tq \exists n\in\N,\ x^n\in I\}.\]

In classical mathematics, if \(A\) is nontrivial \(\Rad(A)\) is the
intersection of all maximal ideals of \(A\), and \(\sqrt{(0)}\) the
intersection of all prime ideals of \(A\).  Notice that an ideal \(I\)
is contained in \(\Rad(A)\) if and only if \(1+I\subseteq A^\times\) and
that \(x\in A\) is invertible if and only if it is invertible modulo
\(\Rad(A)\).

The following classical result is true constructively, when we read
\(x\in A\setminus A^\times\) as ``\(x\in A\) and  \( (x\in
A^\times\Rightarrow False) \)''.

\begin{lemma} If \(A\) is a nontrivial local ring, then
\(\Rad(A)=A\setminus  A^\times\), and it is the unique maximal ideal of
\(A\). We denote it by \(\m_A\) or simply by \(\m\).
\end{lemma}

 The {\em residue field} of a (nontrivial) 
local ring \(A\) with maximal ideal \(\m\) is \(\k=A/\m\). If\/ \(\k\)
is discrete, \(A\) will be called {\em residually discrete}.

A nontrivial ring \(A\) is local and residually discrete if and only
if we have
\[ \forall x\in A\quad (x\in A^{\times }\;\mathrm{or}\;x\in\Rad(A)),
\]
with the constructive meaning of the disjunction.

\mni\textit{Remark.} 
In constructive mathematics, a Heyting field (or simply a field) is a
nontrivial local ring in which ``\(x\) not invertible implies \(x=0\)''.  
This is
the same thing as a nontrivial
local ring whose Jacobson radical is \(0\).

The ring \(A\) defined by \(A=S^{-1}\R[T]\), where \(S\) is the
set of polynomials \(g\) with \(g(0)\in\R^\times\), is  a local
ring: the statement \(\forall x\in A,\ x\in A^\times\hbox{\ or\
}(1+x)\in A^\times\) holds. The residue field of \(A\) is \(\R\), and
the quotient map \(A\dans\R\) is given by \(f/g\mapsto f(0)/g(0)\).
This provides an example of a local ring \(A\) which is \emph{neither}
discrete \emph{nor}
residually discrete.  The ring of \(p\)-adic integers is an example of
a local ring which is residually discrete but \emph{not} discrete.

Some results in this paper avoid the hypothesis
for a ring to be discrete. We think that
when it is possible this greater generality is often usefull,
as shown by the previous examples.

\subsection{Idempotents and idempotent matrices}

\begin{defi} For a commutative ring \(C\) we shall denote \(\B(C)\)
the boolean algebra of idempotents of \(C\).  The operations are: 
\(u\wedge v:=u v,\, u \vee v:= u+v-u v, \,u\oplus 
v:=u+v-2 u v=(u-v)^2\), the complementary of \(u\) is \(1-u\), and the partial
ordering, \(u \preceq v \Longleftrightarrow u \wedge v=u
\Longleftrightarrow u \vee v=v \).
\end{defi}

Note that the partial ordering can be expressed in terms of the
homomorphism \(\mu_{z}\) of multiplication by \(z\) in \(C\): \(u
\preceq v \Longleftrightarrow {\ker}(\mu_{v}) \subseteq
{\ker}(\mu_{u}) \Longleftrightarrow {\Im}(\mu_{u}) \subseteq
{\Im}(\mu_{v}) \).

\smallskip 
A nonzero idempotent \(e\) is said to be \emph{indecomposable} if when it is
written as the sum of two orthogonal
idempotents \(e_{1}\) and \(e_{2}\), then either \(e_{1}=0\) or \(e_{2}=0\).

\smallskip 
A family of idempotent elements
\(\{r_1,\ldots,r_m\}\) in a commutative ring is a \emph{basic system
of orthogonal idempotents} if
\(\sum\nolimits_{i=1}^m r_i=1\) and \(r_i r_j=0\) for
\(1\leq i<j\leq m\).  

\smallskip 
If \(B\) is a finitely generated and discrete boolean algebra, it is
possible to construct a basic system of orthogonal indecomposable
idempotents \(\{r_1,\ldots,r_m\}\) generating \(B\). This shows that
\(B\) is isomorphic to the boolean algebra \(\F_{2}^m\) (where the
field with two elements \(\F_{2}\) is viewed as a boolean algebra).

\begin{lemma}\label{idemnilpot} \emph{(idempotents are always 
isolated)}\\ If \(e,h\) are idempotents and \(e-h\) is in the Jacobson
radical then \(e=h\). In other words, the canonical map \(\B(A)\to
\B(A/\Rad(A))\) is injective.
In particular if \(\B(A/\Rad(A))\) is discrete then so is \(\B(A)\).
\end{lemma}

\begin{proof} First notice that if an idempotent \(f\) is in the Jacobson radical
then \(f=0\) since \(1-f\) is an invertible idempotent.  Now two
idempotents \(e,h\) are equal if and only if \(e\oplus h=0\). But
\(e\oplus h=(e-h)^2\).  So we are done.
\end{proof}

\mni\textit{Remark.} Lemma \ref{idemnilpot} is a sophisticated
rewriting of the identity $(e-h)^3=(e-h)$ when $e$ and $h$ are
idempotents. 

\begin{defi} A commutative ring \(A\) is said to have {\em the property of idempotents lifting}
when the canonical map \(\B(A)\to
\B(A/\Rad(A))\) is bijective.
\end{defi}

\begin{lemma}\label{idemnilpot2} \emph{(idempotents modulo nilpotents 
can always 
be lifted)}\\ The canonical map \(\B(A)\to
\B\left(A/\sqrt{(0)}\right)\) is bijective.
\end{lemma}
\begin{proof} Injectivity comes from Lemma \ref{idemnilpot}. 
    If \(e^2-e=n\) is nilpotent, \eg \(n^{2^k}=0\), then for \(e'=3e^2-2e^3\)
    we have \(e'-e\in nA\) and \((e')^2-e'\in n^2A\). So it is sufficient
    to perform \(k\) times the Newton iteration \(x\mapsto 3x^2-2x^3\).
\end{proof}

\noindent\textit{Remark.} The notion of finite boolean algebra in 
classical mathematics corresponds to several \emph{nonequivalent}\footnote{As
for ``\(\R\) is \emph{not} a discrete field'', this can be proved
by showing that the contrary would imply some principle of
omniscience.}
notions in constructive mathematics. A set \(E\) is said to be
\emph{finite} if there exists a bijection with an initial segment
\([1..n]\) of \(\N\), 
\emph{bounded} if we know a bound on the number of pairwise distinct
elements, \emph{finitely enumerable} if there exists a surjection from
some \([1..n]\) onto \(E\).  Finite sets are finitely enumerable
discrete sets.  Finitely enumerable sets are bounded.  The set of the
monic divisors of a monic polynomial on a discrete field is discrete
and bounded but a priori \emph{not}\footnote{Same thing.} 
finitely enumerable.  A boolean
algebra is finitely enumerable if and only if it is finitely
generated.

\medskip  A \emph{projective module of finite type} over a ring
    \(A\) is a module isomorphic to a direct summand of a free module
    \(A^m\).  Equivalently, \(M\) is isomorphic to the image of an
    idempotent matrix \(F\in A^{m\times m}\).

In the following lemma we introduce a polynomial \(P_{F}(T)\) which is
the determinant of the multiplication by \(T\) in
\(\Im(F)\otimes_{A}A[T]\).

\begin{lemma}\label{idematrix}
    If \(F\in A^{m\times m}\) is an idempotent matrix,  
let \[P_{F}(T):=\det(\I_{m}+(T-1)F)= \sum_{i=0}^me_{i}T^i.\] 
Then  \(\{e_0,\ldots,e_m\}\)  is a basic system of orthogonal
idempotents.
\end{lemma}
 If \(P_{F}(T)=T^r\) the projective module \(\Im\,F\) is said 
to have \emph{constant rank} \(r\).
\begin{proof} A direct computation shows that
\(P_{F}(TT')=P_{F}(T) P_{F}(T')\) and \(P_{F}(1)=1\).
\end{proof}

It can be shown that \(\mathrm{Tr}(F)=\sum_{k=0}^n ke_{k}\), so when
\(\Z\subseteq A\), \(\Im\,F\) has constant rank \(r\) if and only if
\(\mathrm{Tr}(F)=r\).

\subsection{Flat and faithfully flat algebras}

\begin{defi} An \(A\)-algebra \(\varphi :A\to B\) is
\emph{flat} if for every linear form 
\[
\alpha\tq \begin{matrix}  & A^n & \rightarrow & A\cr
 &(x_1,\dots,x_n)& \mapsto & a_1 x_1+\cdots+a_n x_n
\end{matrix}
\]
(\(\alpha\) is given by the row vector \((a_{1},\ldots,a_{n})\)), the
kernel of 
\[\alpha^*\tq\begin{matrix} & B^n & \rightarrow & B\cr 
&(x_1,\dots,x_n)&\mapsto & \varphi(a_1) x_1+\cdots+\varphi(a_n) x_n\cr
\end{matrix}\]
(\(\alpha^*\) is given by the row vector
\((\varphi(a_{1}),\ldots,\varphi(a_{n}))\)) is the
\(B\)-module generated by \(\varphi (\ker \alpha)\).  
\end{defi}

This property is easily extended to kernels of arbitrary matrices.  So
the intuitive meaning of flatness is that the change of ring from
\(A\) to \(B\) doesn't add ``new'' solutions to homogeneous linear systems. 
One says also that \(B\) is flat over \(A\), or \(\varphi\) is a flat
morphism.
 
\begin{example}
The composition of two flat morphisms is flat.  A localization
morphism \(A\to S^{-1}A\) is flat.  If \(B\) is a free \(A\)-module it
is flat over \(A\).
\end{example}

\begin{defi} A flat algebra is \emph{faithfully flat} if for
every linear form \(\alpha :A^n\to A\) and every \(c\in A\) the linear
equation \(\alpha(x)=c\) has a solution in \(A^n\) if the linear
equation \(\alpha^*(y)=\varphi(c)\) has a solution in \(B^n\).
\end{defi}

In this case \(\varphi\) is injective, \(a\) divides \(a'\) in \(A\)
if \(\varphi(a)\) divides \(\varphi(a')\) in \(B\), and \(a\) is a
unit in \(A\) if \(\varphi(a)\) is a unit in~\(B\).

The property in the definition of faithfully flat is easily extended
to solutions of arbitrary linear systems.  So the intuitive meaning of
faithfull flatness is that the change of ring from \(A\) to \(B\)
doesn't add ``new'' solutions to linear systems.
 
\begin{defi} We say that a ring morphism \(\varphi :A\to B\)
\emph{reflects the  units} if for all \(a\in A\), \(\varphi(a)\in
B^{\times} \Rightarrow a \in A^{\times}\).\
\end{defi}

\begin{lemma}\label{fideleplat1}
A flat morphism \(\varphi :A\to B\) is faithfully flat if and only if
for every finitely generated ideal \(\fa\) of \(A\) we have that,
\(1_{B}\in\fa B\Rightarrow  1_{A}\in\fa\).
In case \(B\) is local this means that \(\varphi\) reflects the units.
\end{lemma}

\begin{proof} 
The condition is clearly necessary.  Let \(\fa = \la
a_{1},\ldots,a_{n}\ra\) and \(c\) in \(A\).  The equation \(\alpha(x)=c\) has a solution
in \(A^n\) if and only if \(\fa:c=\la 1\ra\).  Since the morphism is
flat \(\varphi(\fa:c) B=(\varphi(\fa):\varphi(c))\).  If 
\(\alpha^*(y)=\varphi(c)\) has a solution in \(B^n\), 
\(1\in (\varphi(\fa):\varphi(c))\). So we have a finitely many 
 \(x_{j}\in \fa:c\) such
that \(1\in\la (\varphi(x_{j}))_{j=1,\ldots ,k}\ra_{B}\). 
If the condition holds, \(1\in\la (x_{j})_{j=1,\ldots ,k}\ra_{A}\),
so \(1\in (\fa:c)\):
the morphism is faithfully flat.
\end{proof}

\begin{defi}
 A ring morphism \(\varphi\) from a local ring \((A,\m_{A})\) to a
 local ring \((B,\m_{B})\) is said to be \emph{local} when it reflects
 the units.
\end{defi}

This implies also that \(\varphi(\m_{A})\subseteq \m_{B}\).
When \(A\) and \(B\) are residually discrete we have the converse
implication: \(\varphi(\m_{A})\subseteq \m_{B}\). So that the  ring
morphism is local.

A particular case of lemma \ref{fideleplat1} is the classical following one.
It works constructively thanks to the previous ``good'' definitions
in a constructive setting.

\begin{lemma}\label{fideleplat}
A flat morphism between local rings  is local if and only if 
it is faithfully flat.
\end{lemma}

\noi \emph{Remark.} In classical mathematics, an algebra over a field
has always a basis as vector space.  In a constructive setting, this
property can be in general replaced by the fact that a nontrivial
algebra over a discrete field is always faithfully flat.
 
\section{Finite algebras over local rings}

An \(A\)-algebra \(B\) is {\em finite} if it is finite as \(A\)-module.

 \subsection{Preliminaries}

When \(A\) is a discrete field the classical structure theorem for 
finite $A$-algebras, which is a basic tool,
has to be rewritten  to be constructively valid.
This will be done in Corollary~\ref{structthm1}. 

\begin{lemma}[Cayley-Hamilton] (see \cite{Eis95} Theorem 4.3)
\label{charpol}\\
Let \(M\) be a finite module over \(A\).  Let
\(\phi: M \rightarrow M\) an homomorphism such that \(\phi(M)\subseteq
\fa M\) for some ideal \(\fa\) of \(A\). Then we have a polynomial
identity of homomorphisms, 
\[ \phi^n+a_1\phi^{n-1}+ \ldots+a_n\I_M=0  \quad \quad  (*)\]  
where \(a_h \in {\fa^h}\) and  $n$ is the cardinality of
some system of generators of $M$. 
\end{lemma}

\begin{coro}[Nakayama's lemma]\label{Nakayama} 
Let \(M\) be a finite module over a ring \(A\), \(\m\) an ideal, and
\(N\subseteq M\) a submodule. Assume that 
\[M = N+ \m M\,\]
Then there exists \(m\in\m\) such that \((1+m)M\subseteq N\).  If
moreover \(\m\subseteq\Rad(A)\), then \(M=N\).
\end{coro}

Applying Lemma \ref{charpol} to the multiplication by an element in a
finite algebra, we get the  following corollary.

\begin{coro}\label{inverse} 
Let \(\varphi :A\to B\) be a finite algebra (\(B\) is an \(A\)-module generated  by \(n\) 
elements), \(\fa\) an ideal of \(A\), \(A_1=\varphi(A)\) and 
\(\fa_1=\varphi (\fa)\).
\begin{enumerate}
\item Every \(x\in B\) is integral over \(A_1\). If moreover,
\(x\in\fa_1 B\) then \(f(x)=0\) for some \(f(X)=X^n+a_1
X^{n-1}+\cdots+a_n\) where \(a_h \in {\fa_1^h}\).
\item If \(x\in B^\times\),
then there exists \(f\in A_1[X]\) such that \(f(x) x=1_B\) (with
\(\deg(f) \leq n-1\)).
\item  
$A_{1}\cap B^\times=A_{1}^\times$ and $A_{1}\cap\Rad(B) = \Rad (A_{1}) $.
\item Assume  \(B\) is nontrivial. 
If \(A\) is local \(\varphi\) reflects the units.   
If moreover  \(B\) is local and  flat over \(A\)
then it is faithfully flat.
\end{enumerate}
\end{coro}

\begin{proof}
\emph{(2)} Let \(y\) be the inverse of \(x\). If \(y^n+a_1 y^{n-1}+ 
\cdots+a_n=0\) with \(a_{i}\in A\), multiplying by \(x^n\) we get
the result.\\
\emph{(3)} Let  \(x\in A_{1}\cap B^\times \). 
Applying (2) we get \(v=f(x)\in A_{1}\) such that
\(x v =1_{B}\).\\ 
If  \(x\) is in 
$\Rad(B)$, and   \(y\in\A_{1}\), then
$1+xy\in A_{1}\cap B^\times=A_{1}^\times$, so $x\in\Rad(A_1)$.\\
If  $x\in\Rad(A_1)$ and $b\in B$, we have to show that 
$z=-1+xb$ is invertible. Write $b^n+a_{n-1}b^{n-1}+\cdots + a_{0}=0$
with $a_{i}$'s in $A_{1}$. So 
$$(z+1)^n+a_{n-1}x(z+1)^{n-1}+a_{n-2}x^2(z+1)^{n-2}+\cdots + a_{0}x^n=0.$$
The constant coefficient of this polynomial in $z$ is
$1+a_{n-1}x+a_{n-2}x^2+\cdots +a_{0}x^n$, so it is invertible,
and $z$ is invertible.\\
\emph{(4)}
Since \(B\) is nontrival
\(\ker\,\varphi\subseteq \Rad(A) \). If $\varphi(u)\in B^\times$, 
by (3) we have $v\in A$ such that \(uv\in 1+ \ker\,\varphi
\subseteq 1+ \Rad(A)\subseteq
A^\times\).  If moreover \(B\) is local and flat over \(A\) we
conclude by Lemma~\ref{fideleplat1}.
\end{proof}

\begin{defi}
We say that a ring \(B\) is \emph{zero-dimensional} if
\[
\forall x\in B\,\,\exists y\in B\,\,\exists k\in \N,\,\,x^k (1-x y)=0\,.
\]
\end{defi}

Now we get a constructive version of the classical structure theorem.
Notice that it is not assumed that \(B\) has a finite basis over \(\k\).

\begin{coro}\label{structthm1}
    \emph{(structure theorem for finite algebras over discrete 
fields).}\\
    Let \(B\) be a finite algebra over a discrete
    field \(\k\).
\begin{enumerate}
\item \(B\) is \emph{zero-dimensional}, more precisely
\[
\forall x\in B\,\,\exists s\in A[X]\,\,\exists k\in \N,\,\,x^k (1-x s(x))=0\,.
\]
\item \(\Rad(B)=\root B \of {\la 0\ra }\). So \(B\) has the property of idempotents 
lifting.
\item For every \(x\in B \), there exists an idempotent \(e\in
\k[x]\subseteq B\) such that \(x\) is invertible in \(B[1/e]\cong
B\flr{1-e}\) and nilpotent in \(B[1/(1-e)]\cong B\flr{e}\).
\item \(B\) is local if and only if every element is nilpotent or
invertible, if and only if every idempotent is \(0\) or \(1\).
%
Assume \(B\)  is nontrivial, then it is local   if and only if \(\B(B)=\F_{2}\). In 
this case \(B/ \Rad(B)\) is a discrete field. 
\item \(\B(B)\) is \emph{bounded}. 
\item If \(\B(B)\) is \emph{finite}, \(B\)
is the product of a finite number of finite local
algebras (in a unique way up to the 
order of factors).
\end{enumerate}
\end{coro}

\subsection{Jacobson radical of a finite algebra over a local ring}
\label{JRFA}

\setcounter{context}{1}

%
\begin{context} \label{context3}
In Sections \ref{JRFA} and \ref{FAI} \(A\) is a nontrivial
residually discrete local ring with maximal ideal \(\m\) and residue
field \(\k\). We denote by \(a\in A\mapsto \ov a\in\k\) the quotient
map, and extend it to a map \(A[X]\dans \k[X]\) by setting
\(\ov{\sum_i a_i X^i}=\sum{\ov{a_i} X^i}\).
\end{context}

In the sequel we consider finite algebras $B\supseteq A$.
If we had a noninjective homomorphism $\varphi :A\to B$
we could consider  $A_{1}=\varphi (A)\subseteq B$. If $B$
is non trivial
$A_{1}$ is a nontrivial
residually discrete local ring with maximal ideal 
\(\m/\ker \varphi \) and residue
field \(\k\). So our hypothesis  $B\supseteq A$ is not restrictive.

\ms Corollary \ref{structthm1} (1) applied to the \(\k\)-algebra 
\(B/ \m B\) gives the following lemma.

\begin{lemma}\label{zerodim}
Let \(B\supseteq A\) be a finite algebra over \(A\).  
For all \(x\in B\), there exist \(s\in A[X]\) and \(k\in\N\) such that \(x^k (1-xs(x))\in\m B\).
\end{lemma}

\begin{defi} We shall say that a ring \(B\) is \emph{residually zero-dimensional}
if \(B/\Rad(B)\) is zero dimensional, and \emph{semi-local} if moreover
\(\B(B/\Rad(B))\) is bounded. 
If moreover \(B\) has the property of idempotents lifting,
we say that  \(B\) is \emph{decomposable}.
\end{defi}

Notice that our definition of a semi-local ring is equivalent (for nontrivial
rings), in classical mathematics, to the usual one.

In classical mathematics, if \(B\) is \emph{decomposable},
since \(\B(B/\Rad(B))\) is finite, \(B\) is isomorphic to 
a finite product of local rings, \ie it is called a 
\emph{decomposed ring} in \cite{LafonMarot}.

In the following proposition it is not assumed that \(B/ \Rad(B)\)
or  \(B/ \m B\)
have finite bases over~\(\k\).
%
\begin{proposition}\label{radical}  
Let \(B\supseteq A\) be a finite algebra over \(A\). 
\begin{enumerate}
\item \(\Rad(B)=\sqrt{\m B}\). So \(B\) has the property of 
idempotents  lifting if and only if one can lift idempotents modulo \(\m B\). 
\item \(B\) is a semi-local ring.
\item \(B\) is local if and only if  \(B/ \Rad(B)\) is local, if and
only if  \(B/ \m B\) is local.
\item If \(B\) is  local 
then it is residually discrete.
\end{enumerate}
\end{proposition}
%
\begin{proof}
\emph{(1)} Let \(x\in\m B\). Corollary \ref{inverse} (1) implies that 
\(x^{m}+ a_1 x^{n-1}+ \ldots +a_n=0\), with \(a_i \in \m \).  
By Euclidean division, \(x^{n}+ a_1 x^{n-1}+ \ldots +a_n=0= (1-x)
q(x)+ (1+a_1+\cdots +a_n)\) with \(1+a_1+\cdots +a_n \in A^{\times}\).
So \(1-x\in B^{\times}\), and we are done.

Let now \(x\in\Rad(B)\). Lemma \ref{zerodim} implies that \(x^k\in\m
B\).

\noi \emph{(2)} \(B/\Rad(B)\) is a finite \(\k\)-algebra, so 
it is zero-dimensional and its boolean algebra of idempotents is
bounded (see Corollary \ref{structthm1}).

\noi \emph{(3)} A quotient of a local ring is always local.  Let
\(C=B/ \m B\), then \(B/ \Rad(B)=C\big/ \root C\of 0\), so \(B\) and
\(C\) are simultaneously local.  \(B/\Rad(B)\) is a finite
\(\k\)-algebra, so if \(B/ \Rad(B)\) is local, Corollary
\ref{structthm1}~(2) and (4) shows that every element of \(B\) is in
\(\Rad(B)\) or invertible modulo \(\Rad(B)\).  This implies that 
\(B\) is a local ring, and if it is nontrivial, it is residually 
discrete.
\end{proof}
%


\begin{proposition}\label{interadical} Let \(B\supseteq A\) 
    be a finite algebra 
over \(A\), and \(C\subseteq B\) a  subalgebra of \(B\). Then
\(\Rad(C) =\Rad(B) \cap C\).
\end{proposition}
%
\begin{proof} 
This is a particular case of Corollary \ref{inverse}~(3).
%
%
\end{proof}
%

\subsection{Finite algebras and idempotents}\label{FAI}

\begin{lemma}[context \ref{context3}]\label{resultanttrick}
If \(g,h\in A[X]\) are monic polynomials such that \(\ov g\) and \(\ov
h\) are relatively prime, then there exist \(u,v\in A[X]\) such that
\(u g+v h=1\).
\end{lemma}
\def\res{\mathop{\mathrm{res}}}
\begin{proof} Let \(a=\res(g,h)\), the Sylvester resultant of \(g\) and
\(h\). Then \(g\) and \(h\) being monic, \(\ov a = \res(\ov g, \ov
h)\). Then from the hypotheses, we have \(\ov a\neq 0\), that is
\(a\in A^\times\). Now \(a\) can be written \(a=u_0 g+v_0
h\), and we get the result.
\end{proof}

The following proposition is a reformulation of \cite[12.20]{LafonMarot}.  
Our proof follows directly \cite{LafonMarot}. It
is a nice generalization of a standard result in the case where \(A\)
is a discrete field.

\begin{proposition}[context \ref{context3}] \label{factoridem}
Let \(f\in A[X]\) monic. Let \(B\) be the finite \(A\)-algebra
\[B=A[X]\flr{f}=A[x].\]
There is a bijection between the idempotents of \(B\), and
factorizations \(f=g h\) with \(g\), \(h\) monic polynomials and
\(\gcd(\ov g, \ov h)=\ov 1\in\k\). More precisely this bijection
associates to the factor \(g\in A[X]\) the idempotent \(e(x)\in B\)
such that \(\la g(x)\ra=\la e(x)\ra\) in \(B\).
\end{proposition}
\begin{proof} We introduce some notations.  The quotient map \(A[X]\dans
B=A[x]\) will be denoted by \(r(X)\mapsto r(x)\).
The quotient \(B/\m B\) is a finite \(\k\)-algebra,
isomorphic to \(\k[X]\flr{\ov{f}}\). We denote by \(\ov x\) the class
of \(x\) modulo \(\m B\). The quotient map from \(B\) to
\(B/\m B\) is denoted by \(r(x)\mapsto \ov{r(x)}=\ov r(\ov
x)\). The canonical map from \(\k[X]\) to \(B/\m B\) is
denoted by \(\ov r(X)\in \k[X]\mapsto \ov{r(x)}=\ov r(\ov x)\).

The situation is summed-up in the following commutative diagram:
\def\darrow{\Big\downarrow}
\[\begin{CD} 
         r=r(X)\in A[X] @>>> r(x)\in B \\
                   @VVV      @VVV\\
\ov r=\ov r(X)\in \k[X] @>>> \ov r(\ov x)\in B/\m B 
\end{CD}\]

Let \(g,h\in A[X]\) such that \(f=g h\) and \(\gcd(\ov g,\ov
h)=\ov 1\in \k\).  Then thanks to Lemma \ref{resultanttrick}, we have
\(u,v\in A[X]\) such that \(u g+v h=1\).  Let \(e=u\cdot
g\); then \(e-e^2=e (1-e)=u g v h=u v\cdot
f\), and \(e(x)-e(x)^2=u(x) v(x) f(x)=0\); \(e(x)\) is an
idempotent of \(B\).  Note that \(g=e g+v f\) and
\(g(x)=e(x) g(x)\).  So 
\(\la e,f \ra=\la g \ra\) in \(A[X]\), \(\la \ov
e, \ov f\ra=\la \ov g\ra\) in \(\k[X]\) and \(\la g(x)\ra=\la e(x)\ra\) in
\(B\).

\smallskip 
Now assume that we have \(e(X)\in A[X]\), such that \(e(x)^2=e(x)\).

Let \(g_1\) and \(h_1\) be monic polynomials such that \(\ov
{g_1}=\gcd(\ov e,\ov f)\) and \(\ov {h_1}=\gcd(\ov{1-e},\ov f)\). The
polynomials \(\ov e\) and \(\ov{1-e}\) are relatively prime, and \(\ov
f\) divides \(\ov e (\ov{1-e})\), so \(\gcd(\ov {g_1},\ov {h_1})=\ov
1\) and \(\ov f=\ov {g_1} \ov {h_1}\). Let \(\deg {g_1}=n\), \(\deg
h_1=m\); we have \(\deg f=n+m\).

Since $\la \ov {g_1}\ra=\la \ov e, \ov f\ra$, we get 
$\la \ov {g_1}(\ov x)\ra=\la \ov e(\ov x)\ra$. Similarly
$\la \ov {h_1}(\ov x)\ra=\la (\ov{1- e})(\ov x)\ra$.

Now let \(g_2=e g_1\) and \(h_2=(1-e) h_1\). We have
\(\ov{g_1}(\ov x)\in\la\ov{e}(\ov x)\ra\), and \(\ov{e}(\ov x)\in
B/\m B\) is an idempotent, so that \(\ov{g_2}(\ov
x)=\ov{g_1}(\ov x)\). In the same way, we have \(\ov{h_2}(\ov
x)=\ov{h_1}(\ov x)\).

Let
\[  u_i=X^{i} g_2,  \quad i=0,\ldots,m-1 \]
and
\[ v_i=X^{i} h_2,  \quad i=0,\ldots,n-1. \]
The determinant of 
\(\left(\ov{u_0}(\ov x),\dots,\ov{u_{m-1}}(\ov x),\ov{v_0}(\ov x),\dots,\ov{v_{n-1}}(\ov x)\right)\)
w.r.t. the canonical basis $(\ov 1,\ov x,\ldots,\ov x^{n+m-1})$ is invertible
(the matrix is the Sylvester matrix of $\ov{g_1}(\ov x)$ and $\ov{h_1}(\ov x)$). So the determinant of  \(\left({u_0(x)},\allowbreak\dots,\allowbreak{u_{m-1}(x)},
{v_0(x)},\allowbreak\dots,{v_{n-1}(x)}\right)\) w.r.t. the canonical basis $(1,x,\ldots,x^{n+m-1})$ is invertible and the family generates \(B\) as an
\(A\)-module.

Let \(B_1=u_0(x) A+\cdots+u_{m-1}(x) A\) and
\(B_2=v_0(x) A+\cdots+v_{n-1}(x) A\). We have \(B=B_1+B_2\). Now
\(g_2(x)\in\la e(x)\ra\), so \(B_1\subseteq\la e(x)\ra\), and in the
same way \(B_2\subseteq\la 1-e(x)\ra\). We deduce that \(B_1=\la e(x)
\ra\) and \(B_2=\la 1-e(x)\ra\).

Moreover $B_1\subseteq \la g_2(x)\ra \subseteq\la e(x)\ra$, so $B_1= \la g_2(x)\ra =\la e(x)\ra$. Similarly \(B_2=\la h_2(x)\ra=\la 1-e(x)\ra\)

Since \(x^m g_2(x)\in \la e(x)\ra = B_1\) there are \(a_0,\dots,a_{m-1}\in A\)
such that \(x^m g_2(x)=a_0 g_2(x) +\cdots+ a_{m-1}\cdot
x^{m-1} g_2(x)\). Let \(h(X)=X^m-\sum_{i=0}^{m-1} a_i X^i\). We have
\(h(x) g_2(x)=0\). So $\ov f=\ov {g_1} \ov {h_1}$ divides $\ov h  \ov{g_2}=\ov h  \ov{g_1}$. Since $\deg(\ov h)=\deg(\ov {h_1})$ this implies $\ov h=\ov {h_1}$. Moreover, since $\la g_2(x)\ra =\la e(x)\ra$, we get $h(x) e(x)=0$, i.e., $h(x)=(1-e(x)) h(x)$.

In the same way we find a monic polynomial \(g(X)\) of degree \(n\),
such that \(g(x) h_2(x)=0\), $g(x)=g(x) e(x)$ and $\ov g=\ov {g_1}$.

Then \(g(x) h(x)=g(x) h(x) e(x)\cdot(1-e(x))=0\)  in \(B\), so that \(f(X)\) divides
\(g(X) h(X)\).  These polynomials are monic with same degree, so
\(f=g h\).

Note that \(\ov g=\ov {g_1}=\gcd(\ov e, \ov f)\), which shows that the two
mappings  we defined between the set of idempotents and the
factors of \(\ov f\) are  each other's inverse.
\end{proof}

\begin{lemma}[context \ref{context3}]\label{idemlift} Let \(B\supseteq A\) be a finite algebra over
\(A\),  and \(C\subseteq B\) a subalgebra of \(B\). If we have
\(e\in C\) and \(h\in B\) such that \(e^2-e\in\m C\),
\(h-e\in\m B\) and \(h^2=h\), then \(h\) is in~\(C\).
\end{lemma}

\begin{proof} Let \(C_1=C+h C\subseteq B\).    
We have \(h-e\in\Rad(B)\cap C_1=\Rad(C_1)\) by 
Propositions \ref{radical} and \ref{interadical}. 
Since \(h\) and \(e\) are idempotent in
\(C_1/\m C_1\), 
and since $\Rad(C_1/\m C_1)=\Rad(C_1)/\m C_1 $,
Lemma \ref{idemnilpot} implies that \(h=e+z\) for
some \(z\in \m C_1\).  Therefore \(C_1= C+\m C_1\).
Moreover \(\m\subseteq\Rad(B)\cap C=\Rad(C)\) by 
Propositions \ref{radical} and \ref{interadical}. So by
Nakayama's lemma, \(C=C_1\).
\end{proof}

\noindent \emph{Remark.} The preceding lemma will be used in the proof of 
Proposition \ref{monicfactorlift}, where it works as a substitute of \cite[12.23]{LafonMarot}.

\cite[12.23]{LafonMarot} is the following result: if 
\(C\subset B\), with \(B\) integral over \(C\) and if \(B\) is 
decomposed (\ie is a finite product
of local rings), then so is \(C\).


\cite{LafonMarot} use freely (being in classical mathematics) the fact that
a bounded Boolean algebra is finite. This allows to develop a theory
of Henselian local rings based on the notion of decomposed rings. We did not try to develop
a completely parallel development based on the notion of
decomposable rings.

Since there was no need of a result as general as \cite[12.23]{LafonMarot}, we have preferred
to give Lemma \ref{idemlift} with its short constructive proof.


\section{Universal decomposition algebra}

In this section \(A\) is a ring, not necessarily local. The material
presented here will be useful later, in the case of Henselian local
rings.

\def\D{D_A}

\begin{defi} In the ring \(A[X_1,\dots,X_n]\), the {\em elementary
symmetric functions} \(S_1,\dots,S_n\) are defined to be
\[S_k=S_k(X_1,\dots,X_n) = \sum_{1\leq i_1<\cdots<i_k\leq n} X_{i_1}\cdots
X_{i_k} \]
\end{defi}

\begin{defi} Let \(f(T)=T^n+a_1 T^{n-1}+\cdots+a_{n-1} T+a_n\in
A[T]\) a monic polynomial.  The {\em universal decomposition algebra}
of \(f\) is \(\D(f)\) defined by
\[\D(f) = A[X_1,\dots,X_n] \flr{S_1+a_{1}, S_2-a_{2}, \dots,
S_n+(-1)^{n-1} a_0}\]
\end{defi}

We shall denote \(x_{i}\) the class of \(X_{i}\) in \(\D(f)\).
The following result is standard.

\begin{lemma} The universal decomposition algebra  \(\D(f)\) of
\(f \in A[T]\),  is a free \(A\) module of rank \(n!\). A basis of it 
is given by the power products
\[\{x_1^{m_1} \cdots x_n^{m_n}:  0\leq m_{j} \leq j-1; j=1, \ldots,
n\}\]
\end{lemma}

Let \(\S_n\) be the \(n\)-th permutation group. It acts on
\(A[X_1,\dots,X_n]\) by setting \(\sigma X_i=X_{\sigma(i)}\), so
\(\sigma f(X_1,\dots,X_n)=f(X_{\sigma(1)},\dots, X_{\sigma(n)})\). We
have clearly \(\sigma(f g)=\sigma f \sigma g\) and
\(\sigma(f+g)=\sigma f+\sigma g\); if \(\deg f=0\), \(\sigma f= f\).

This group action leaves the ideal \(\la S_1+a_{1}, S_2-a_{2},
\dots, S_n+(-1)^{n-1} a_n\ra\) invariant, so it induces a group
action of \(\S_n\) on \(\D(f)=A[x_1,\dots,x_n]\), so that \(\sigma
f(x_1,\dots,x_n)=f(x_{\sigma(1)},\dots, x_{\sigma(n)})\).

The so called Theorem of Elementary Symmetric Polynomials is the
following lemma.

\begin{lemma} \label{fixSym}
If \(g\in A[X_1,\dots,X_n]\) is fixed by the action of \(\S_n\), then
\(g\in A[S_1,\dots,S_n]\) (the subring generated by the elementary
symmetric functions).
\end{lemma}

In that case, the image of \(g\) in the quotient \(\D(f)\) is in \(A\).
 
However when we consider the induced action in the quotient ring
\(\D(f)\), it can happen that some \(g\in\D(f)\) is fixed by this  
action, but it is not in \(A\), as it is shown with the following
example.

\mni\textit{Example}. Let \(A=\F_2(u)\) (here \(\F_2\) is the Galois
field with two elements), and \(f(T)=T^2-u\). Then
\(\D(f)=A[X_1,X_2]\flr{X_1^2-u, X_2-X_1}\).  Then the whole \(\D(f)\)
is fixed by~\(\S_2\).

\mni Nevertheless we have the useful following result.

\begin{lemma}\label{idemfix} Assume \(A\) has only \(0\) and \(1\) as 
idempotents, then every idempotent \(e\in \D(f)\) invariant by the
action of \(\S_n\) belongs to \(A\).
\end{lemma}
\begin{proof} 
Let \(E \in A[X_1, \ldots, X_n]\), be such that its image in \(\D(f)\)
is \(e\). By the elementary symmetric functions theorem
\(E^{*}:=\prod_{\sigma\in \S_n} E^{\sigma} \in A[S_1,\dots,S_n]\).  We
call \(e^{*}\) the image of \(E^{*} \) in \(\D(f)\), we get
\(e^{*}= \prod_{\sigma\in \S_n}e^{\sigma}=e^{n!}=e\), and we are done. 
\end{proof}


When \(A\) is discrete, so are \(\B(A)\), \(\D(f)\) and
\(\B(\D(f))\). Here is a more subtle result.

\begin{lemma}\label{booldisc} Assume \(A\neq 0\) has only \(0\) and \(1\) as 
idempotents. Then   \(\B(\D(f))\) is discrete.
\end{lemma}

\begin{proof} 
Let \(e\in\B(\D(f))\). Consider the matrix \(F\) representing the
multiplication by \(e\). Then \(e=0\) if and only if the projective
module \(e \D(f)\) has constant rank \(0\), which is equivalent
to \(P_{F}(T)=1\). And we know from Lemma \ref{idematrix} that
\(P_{F}(T)=T^r\) for some \(r\).
\end{proof}

A similar proof shows more generally that if \(\B(A)\) is discrete
then  \(\B(\D(f))\) is discrete.

\begin{defi} An idempotent in \(\D(f)\) is said to be a
\emph{Galois idempotent} when its orbit is a basic system of
orthogonal idempotents. 
\end{defi}

\begin{lemma}\label{bool} Assume \(A\) has only \(0\) and \(1\) as 
idempotents, and let   \(e\in\D(f)\) be an idempotent. Then there
exists a Galois idempotent \(h\in\D(f)\) such that  
\(e\) is a sum of conjugates of~\(h\).
\end{lemma}

\begin{proof} 
Thanks to Lemma \ref{booldisc} the boolean algebra \(\B(\D(f))\) is
discrete.  Let \(B\) be the boolean subalgebra of \(\B(\D(f))\)
generated by the orbit of \(e\) under \(\S_{n}\).  Since \(B\) is
discrete and finitely generated we can find an indecomposable element
of \(B\).  Let \(h_{1}\) be such a minimal nonzero element of \(B\).
For every \(g\in B\), we have \(g h_{1}=0\) or \(h_{1}\).  In
particular the orbit \(h_{1},\ldots,h_{m}\) of \(h_{1}\) is made of
orthogonal idempotents.  So the sum of this orbit is a nonzero
idempotent of \(A\), necessarily equal to 1 by Lemma \ref{idemfix}.
So \(h_{1}\) is a Galois idempotent and for every \(g\in B\),
\(g=\sum_{i=1}^m g h_{i}=\sum_{g h_{i}\neq 0}h_{i}\).
\end{proof}


\section{Henselian Local Rings}
\label{SecHenselian}

\begin{defi}\label{defHensel} 
Let \(A\) be a local ring with maximal ideal \(\m\). We
say that \(A\) is {\em Henselian} if every monic polynomial
\(f(X)=X^n+\cdots+a_1 X+a_0\in A[X]\) with \(a_1\in A^\times\) and
\(a_0\in\m\) has a root in \(\m\).
\end{defi}

It is easy to show that if \(f(X)=X^n+\cdots+a_1 X+a_0\) with
\(a_1\in A^\times\) and \(a_0\in\m\) has a root in \(\m\), then this
root is unique (see Proposition \ref{hensel-nonmonic}).

\begin{context} \label{context5}
In the whole Section \ref{SecHenselian} \(A\) will 
be a nontrivial Henselian local ring with maximal ideal \(\m\) and
residue field \(\k\). We denote by \(a\in A\mapsto \ov a\in\k\) the
quotient map, and extend it to a map \(A[X]\dans \k[X]\) by setting
\(\ov{\sum_i a_i X^i}=\sum{\ov{a_i} X^i}\).
\end{context}

Note that a nontrivial quotient ring of a Henselian local ring is
also a Henselian local ring, with same residue field.
So, as noticed at the beginning of Section \ref{JRFA}, we can
restrict our attention to $A$-algebras containing $A$.

\ms Definition \ref{defHensel} says that it is possible to lift
in the local Henselian ring
a residual simple root of a monic polynomial. 
In this section we show that more general liftings are available
in a Henselian local ring.

Paragraph \ref{AFUG} is devoted to the lifting of any residual simple root
of any univariate polynomial.

In Paragraph \ref{subsecFact} we prove the lifting of coprime factorizations of a monic univariate polynomial and the lifting of idempotents in finite $A$-algebras.

In Paragraph \ref{subsecFact2} we prove the lifting of coprime factorizations of any univariate polynomial.

\subsection{A first useful generalization} \label{AFUG}

\begin{defi} We shall denote \(A(X)\) the Nagata localization of
\(A[X]\), \ie the localization with respect to the monoid of primitive
polynomials (a polynomial is primitive when its coefficients generate
the ideal \(\la 1 \ra\)). It is well known that \(A[X]\subseteq
A(X)\).
\end{defi}

\begin{lemma}[context \ref{context5}] Let \(f(X)=a_n X^n +\cdots+a_1 X+a_0\), with
\(a_1\in A^\times\) and \(a_0\in\m\). There exists a monic polynomial
\(g(X)\in A[X]\), \(g(X)=X^n+\cdots+b_1 X+b_0\), with \(b_1\in
A^\times\) and \(b_0\in\m\), such that the following equality holds in
\(A(X)\):
\[a_0 g(X)=(X+1)^n f\left( {-a_0 a_1^{-1}\over X+1}
\right).\] 
\end{lemma}
\begin{proof}
We have 
\[\begin{aligned}
X^n f\left(\frac{-a_{0}\cdot
a_{1}^{-1}}{X}\right) & = a_{0} \left(X^n-X^{n-1}+a_{0}\cdot
\sum_{j=2}^{n}(-1)^j a_{j}a_{0}^{j-2} a_{1}^{-j} \cdot
X^{n-j}\right)\\
 & = a_{0} h(X)
\end{aligned}\]
with
\[ h(X)= X^n-X^{n-1}+a_{0}\cdot
\sum_{j=2}^{n}(-1)^j   a_{j}a_{0}^{j-2} a_{1}^{-j}  X^{n-j}
= X^n-X^{n-1}+a_{0} \ell(X) \]
We let \(g(X)=h(X+1)=X^n+\cdots+b_1 X+b_0\). It is a monic polynomial,
with constant term \(b_0=g(0)=h(1)=a_{0} \ell(1)\in\m\), and linear
term \(b_1=g'(0)=h'(1)=1+a_{0} \ell'(1)\in1+\m\).
\end{proof}

\begin{proposition}\label{hensel-nonmonic}
Let \(f(X)=a_n X^n +\cdots+a_1 X+a_0\), with \(a_1\in A^\times\)
and
\(a_0\in\m\). Then \(f\) has a unique root in \(\m\).
\end{proposition}
\begin{proof} First we prove the uniqueness. Let us write $f(X+Y)=f(X)+Y(f'(X)+Y f_2(X,Y))$. If $f(a)=f(a+\mu)$ with $\mu\in m$ and $f'(\mu)\in A^\times$, we replace $X$
by $a$ and $Y$ by $\mu$ and we get $b \mu=0$ with $b\in A^\times$, so $\mu=0$.

Let \(g(X)\) be the polynomial associated to \(f\) by the
previous lemma, and \(\alpha\in\m\) its root. Then \((1+\alpha)\in
A^\times\); we put \(\beta={-a_0 a_1^{-1}\over \alpha+1}\), and we have
\(-a_0 g(\alpha)=(\alpha+1)^n f\left(\beta\right)\), so that
\(f(\beta)=0\).
\end{proof}

\begin{coro}[context \ref{context5}]\label{hensel-roots}
Let \(f(X)=a_n X^n +\cdots+a_0\in A[X]\) such that \(\ov f(X)\in\k[X]\) has a
simple root \(a\in\k\). Then there exists a unique root \(\alpha\in A\)
of \(f\) such that \(\ov\alpha=a\).
\end{coro}
\begin{proof}
Replace $X$ by $X+\gamma$ where $\ov \gamma=a$ and use the previous proposition.
\end{proof}
%

\subsection{Universal decomposition algebra over a Henselian local ring}

Let \(f(T)=T^n+\cdots+a_1 T+a_0\in A[T]\) be a monic polynomial of degree
\(n\), and let \(D=\D(f)=A[x_1,\dots,x_n]\) be its universal decomposition
algebra.

It is easy to check that \(D/\m D\) is (isomorphic to) \(\Dk(\ov
f)\). The permutation group \(\S_n\) acts both on \(D\) and \(D/\m D\).

For every \(r(\X)=r(X_1,\dots,X_n)\in A[X_1,\dots,X_n]\) we denote by
\(r(\x)=r(x_1,\dots,x_n)\) its image in \(A[x_1,\dots,x_n]=D\), and
its image under the quotient map \(D\dans D/\m D\) is
\(\ov{r(\x)}=\ov{r}(\ov\x)\). The canonical map from \(\k[\X]=\k[X_1,\dots,X_n]\) to
\(D/\m D\) is denoted by \(\ov r(\X)\mapsto \ov{r(\x)}=\ov{r}(\ov\x)\).

The situation is summed-up by the following commutative diagram.
\[\begin{CD} 
         r=r(\X)\in A[\X] @>>> r(\x)\in D \\
             @VVV                   @VVV \\
\ov r=\ov r(\X)\in \k[\X] @>>> \ov{r(\x)}\in D/\m D 
\end{CD}\]

In Proposition \ref{idemlift2} we show that \(D\) 
admits lifting of idempotents modulo
\(\m D\).
This has tight connection with \cite[12.27]{LafonMarot}, 
which settles that  \(D\) is a finite product of local rings.
The result in \cite{LafonMarot} cannot be reached constructively, but  
Proposition \ref{idemlift2} implies that  \(D\) 
is decomposable, and Proposition 
\ref{extension2} tells us that if
\(\B(D)\) is finite (we only know it is bounded) then  
 \(D\) is a finite product of local rings. 
So we see that we will get finally good constructive versions of \cite{LafonMarot}'s 
result, but the general organization of the material is slightly 
different. 
 
Propositions \ref{idemgaloislift} and \ref{idemlift2}
give a useful constructive substitute 
for  \cite[12.27]{LafonMarot}.
Their proofs can be seen as extracting the constructive content of
the proof of \cite{LafonMarot}.

\begin{proposition}[context \ref{context5}, lifting a Galois idempotent of $D/\m D$]\label{idemgaloislift} ~

\noindent Let \(r(\X)\in A[\X]\) be such
that \(\ov{r(\x)}\in D/\m D\) is a Galois idempotent. Then there
exists \(e(\X)\in A[\X]\) such that \(e(\x)\) is a Galois idempotent of \(D\) with
\(\ov{e(\x)}=\ov{r(\x)}\).
More precisely, if the orbit of $\ov{r(\x)}$ is $[\ov{r(\x)},\sigma_2(\ov{r(\x)}),\ldots,\sigma_h(\ov{r(\x)})]$, then the orbit of \(e(\x)\) is $[e(\x),\sigma_2(e(\x)),\ldots,\sigma_h(e(\x))]$.
\end{proposition}

\def\stab{\mathop{\mathrm{Stab}}}

\begin{proof} 
By Lemma \ref{booldisc} \(\B(D_{\k}(\ov f))\) is discrete. Let
\(\ov{r_1(\x)}=\ov{r(\x)},\allowbreak\ov{r_2(\x)},\allowbreak\dots,
\allowbreak\ov{r_h(\x)}\) be the orbit of \(\ov{r(\x)}\) under the
action of \(\S_n\); let \(\sigma_2,\dots,\sigma_h\in\S_n\) such that
\(\ov{r_i(\x)}=\sigma_i\ov{r(\x)}\).

Let \(G=\stab_{\S_n}\l(\ov{r(\x)}\r)=
\l\{\sigma\in\S_n\tq\sigma\ov{r(\x)}=\ov{r(\x)}\r\}\).


Let \(c_1(\X)=\prod_{\sigma\in G} \sigma r(\X)\). Then \(\ov{c_1(\x)}=
\l(\ov{r_1(\x)}\r)^{|G|}=\ov{r_1(\x)}\).

For \(i=2,\dots,h\), let \(c_i(\X)=\sigma_i c_1(\X)\). We have
\(\ov{c_i(\x)}=\ov{r_i(\x)}\).

Let \(P(T)=\prod_{i=1}^h (T-c_i(\X))\in A[\X][T]\). The coefficients of
\(P(T)\) are invariant under the action of \(\S_n\); so \(P(T)\in
A[S_1,\dots,S_n][T]\). We write \(P(T)=R(S_1,\dots,S_n,T)\).

So let \(p(T)=R(s_1,\dots,s_n,T)\), where 
\(s_{i}=S_{i}(\x)=(-1)^ia_{n-i}\).

We get \(p(T)=\prod_i (T-c_i(\x))\in A[T]\) 
(remember that \(f(T)=T^n+\cdots+a_1 T+a_0\)). Modulo \(\m\), we have
\(\ov p(T)=\l(T-\ov{r_1(\x)}\r)\cdots\l(T-\ov{r_h(\x)}\r)=T^h-T^{h-1}\in\k[T]\).

So \(\ov p(T)\in\k[T]\) admits \(1\in\k\) has a simple root. We can lift
it to a root \(\alpha\in A\) of \(p\), such that \(\ov\alpha=1\). We have
\(\ov{p'(\alpha)}=1\), so that \(p'(\alpha)\in A^\times\); let \(\lambda\in
A\) be its inverse (we have \(\ov\lambda=1\)).

Let \(e_i(\x)=\lambda \prod_{j\neq i}(\alpha-c_j(\x))\in D\).

We have, for \(i\neq k\), \(e_i(\x) e_k(\x)=0\), and \(\sum
e_i(\x)=\lambda p'(\alpha)=1\); hence \(e_i(\x)^2=e_i(\x)\). Moreover,
\(\ov{e_i(\x)}=\prod_{j\neq i}(1-\ov{r_j(\x)})=\ov{r_i(\x)}\).
\end{proof}

\begin{proposition}\label{idemlift2} \emph{(lifting an arbitrary idempotent of $D/\m D$)}

\noindent  Let \(r(\X)\in A[\X]\) be such 
that
\(\ov{r(\x)}\in D/\m D\) is an idempotent. Then there exists
\(e(\X)\in A[\X]\) such that \(e(\x)\) is an idempotent of \(D\) with
\(\ov{e(\x)}=\ov{r(\x)}\).
\end{proposition}

\begin{proof} Lemma \ref{bool} says that $\ov{r(\x)}$ is a sum of conjugates
of a Galois idempotent of $D/\m D$.
Proposition \ref{idemgaloislift} allows us to lift this Galois idempotent.
The corresponding sum of conjugates is an idempotent
which is a lifting of $\ov{r(\x)}$.
\end{proof}

\subsection{Fundamental theorems}\label{subsecFact}

\setcounter{context}{2}
\begin{context} \label{context5b}
In Sections \ref{subsecFact} 
and \ref{subsecFact2} \(A\) will 
be a nontrivial Henselian residually discrete local ring with maximal ideal \(\m\) and
residue field \(\k\).
\end{context}

\begin{proposition}\label{monicfactorlift} Let \(f,g_0,h_0\in A[T]\) be 
monic
polynomials, such that \(\ov f = \ov g_0 \ov h_0\) in \(\k[T]\) and
\(\gcd(\ov g_0,\ov h_0)=1\). Then there exist monic polynomials \(g,h\in
A[T]\) such that \(f=g h\) and \(\ov g=\ov g_0\), \(\ov h=\ov h_0\).
Moreover this factorization is unique.
\end{proposition}

\begin{proof} Let \(B=A[t]=A[T]\flr{f(T)}\). From Proposition 
\ref{factoridem}
we see that it is enough to show that given an idempotent \(\ov e(t)\in
B/\m B\), one can lift it to an idempotent \(e'(t)\in B\).

Let \(D=\D(f)\). It is an extension ring of \(B\). The situation is the
following:

\def\darrow{\vtop to 10pt{%
\vss%
\hbox{\(\left\downarrow\vrule width 0pt depth 6pt height 0pt\right.\)}%
\kern6pt\vss}}
\[\begin{CD}
    A @>>> B               @>>> D \\
 @VVV      @VVV                 @VVV\\
   \k @>>> B/\m B @>>> D/\m D 
\end{CD}\]

Thanks to Proposition \ref{idemlift2} there exists an idempotent \(e'\in
D\) such that \(e-e'\in\m D\). Then Lemma \ref{idemlift} shows that \(e'\in
B\), and we are done.

The unicity comes from Proposition \ref{factoridem},
Lemma \ref{idemnilpot} and Proposition \ref{radical}~(1). 
\end{proof}

We can now lift idempotents in all finite \(A\)-algebras.
With Theorem \ref{localImpliqueHenselian} 
this is the main result of the paper.

\begin{theorem}[context \ref{context5b}]\label{idemlift3} Any finite
\(A\)-algebra \(B\supseteq A\) 
has the property of idempotents lifting. 
More precisely for any \(e\in B\) such that \(e^2-e\in \m B\)
we can construct an idempotent \(e'\in B\) such that \(e-e'\in \m B\).
\end{theorem}

\noindent 
\emph{Remark}. Since \(B\) is semi-local (Proposition \ref{radical}~(2)), it is decomposable. So we get in classical mathematics the theorem that any finite
\(A\)-algebra is decomposed (\cite{LafonMarot}).  
\begin{proof} We denote by \(b\in B\mapsto \ov b\in
    B/\m B\) the quotient
map.  Using finiteness of \(B\), we find a monic polynomial \(F(T)\in
A[T]\) such that \(F(e)=0\). Its image modulo \(\m A[T]\) is \(\ov
F(T)\in\k[T]\).

Now \(\ov F(\ov e)=\ov 0\) in \(B/\m B\) which is a finite \(\k\)-algebra,
and \(\ov e^2=\ov e\).

We write \(\ov F(T)= T^\ell (T-1)^k   H(T)\) with
\(k,\ell\geq 0\) and \(H\) prime with
\(T\) and \(T-1\). 
If $\ell=0$, $\ov e$ is invertible, so $\ov e= \ov 1$ is lifted as $1$.
Similarly, if $k=0$, $\ov e= \ov 0$ is lifted as $0$.
If $\ell>0,\,k>0$, using Proposition \ref{monicfactorlift}, we lift the
factorization to \(F=a b h\), with 
\(\ov a=T^\ell\), \(\ov b=(T-1)^k\), \(\ov h=H\). 
Since $T(T-1)U+HV=1$ and \(\ov e^2-\ov e=\ov 0\), $\ov{h(e)}=H(\ov e)$
is invertible, so $h(e)$ is invertible, and $F(e)=0$ implies
$a(e)b(e)=0$. Moreover $\ov{a(e)}=\ov{e}^\ell=\ov e $ 
and $\ov{b(e)}=\ov{(1-e)}^k=\ov 1- \ov e$ (since $k,\,\ell>0$).
So $a(e)+b(e)=\mu\in 1+\m B$, which has an inverse $\nu\in 1+\m B$.
Finally $\nu a(e)$ and  $\nu b(e)$ are complementary idempotents
with $\ov{\nu a(e)}= \ov \nu \ov{a(e)}=\ov e$.
\end{proof}

We have also an easy converse result (\eg using Proposition 
\ref{factoridem},
but a direct proof is simpler). 

\begin{proposition} \label{Propreciproque}
    Let \(B\) be a nontrivial residually discrete local ring
such that every finite \(B\)-algebra has the property of idempotents
lifting. Then \(B\) is Henselian.
\end{proposition}

We get now the basic ingredient for the construction of the 
strict Henselization
of a residually discrete local ring.
\begin{theorem}[context \ref{context5b}]\label{localImpliqueHenselian}
Every nontrivial finite local \(A\)-algebra \(B\)   
is a Henselian residually discrete local ring.
\end{theorem}

\begin{proof} 
By Proposition \ref{radical}~(4) \(B\) is residually discrete.  Let
\(C\) be a finite \(B\) algebra; it is a finite
\(A\)-algebra as well. So by Theorem \ref{idemlift3} 
it admits lifting of idempotents modulo
\(\m C=(\m B) C\). Hence
by Proposition \ref{Propreciproque} \(B\) is Henselian.
\end{proof}

The following corollary gives some precision in a particular case.

\begin{coro}\label{extension1}
Let \(f(X)\in A[X]\) be a monic polynomial such that \(\ov
f(X)\in\k[X]\) is  (a power of) an irreducible $h(X)\in\k[X]$.  Let \(B\)
be the quotient algebra \(B=A[x]=A[X]\flr{f(X)}\);
\(B\) is a local Henselian ring with residue field \(\k[X]\flr{h(X)}\).
\end{coro}

\begin{proof} By Proposition \ref{radical} (3) \(B\) is local,
    so apply Theorem \ref{localImpliqueHenselian}.
\end{proof}

Here we get, within precise constructive hypotheses, the
analogue of the characterization of Henselian rings in \cite{LafonMarot}
as local rings satisfying ``every finite algebra is
a finite product of local rings''.

\begin{proposition}\label{extension2}  
Let \(B\) be a finite algebra over \(A\).  Assume that the Boolean algebra \(\B(B/\m B)\) is finite (a priori, we only know it is bounded).  Then
\(B\) is a finite product of local Henselian rings.
\end{proposition}
\begin{proof} By Corollary \ref{structthm1} \(B/ \m B\) is 
a finite product of finite local \(\k\)-algebras.  We lift the
idempotents by Theorem \ref{idemlift3} and we conclude by Proposition
\ref{radical}~(3) and Theorem~\ref{localImpliqueHenselian}.
\end{proof}

\subsection{Factorization  of non-monic polynomials}\label{subsecFact2}

~

\ss Now we turn to the case of non-monic polynomials. 
We want to lift a residual factorization in two coprime polynomials
when one residual factor is monic.
Since the polynomial we hope to factorize is not monic we cannot apply 
Proposition~\ref{monicfactorlift}.
%
\begin{lemma}\label{IfInfinite} Let \(f,g_0,h_0\in A[X]\) such that
\(\ov f=\ov g_0 \ov h_0\) with \(\gcd(\ov g_0,\ov h_0)=1\)
and \(g_0\) is monic.
If  \(f(0)\in A^\times\), then there exist \(g,h\in A[X]\) with
\(g\) monic, such that \(f=g h\), \(\ov g=\ov g_0\) and \(\ov
h=\ov h_0\). Moreover this factorization lifting is unique.
\end{lemma}

\begin{proof} If \(f(X)\) is monic, this is Proposition
\ref{monicfactorlift}.

If \(f(X)\) is not monic, then let \(d=\deg f\) and
\(p(X)=f(0)^{-1} X^d f(1/X)\). Let \(n=\deg g_0\) and
\(q_0(X)=X^n g_0(1/X)\). Then \(\ov q_0\) divides \(\ov p\),
which is monic; if we write \(\ov p=\ov q_0 \ov r_0\), we have
\(\ov r_0(X)=X^{d-n}\ov h_0(1/X)\), so that \(\gcd(\ov q_0,\ov
r_0)=1\). By Proposition \ref{monicfactorlift}, we find \(q,r \in
A[X]\) such that \(p=q r\) and \(\ov q=\ov q_0\).

Let \(g(X)=1/X^n q(X)\). Then \(g(X)\) divides \(f(X)\) and \(\ov
g=\ov g_0\). We let \(h\in A\) be such that \(f=g h\).

The unicity comes from the unicity in Proposition~\ref{monicfactorlift}.
\end{proof}

We drop the extra-hypothesis ``\(f(0)\in A^\times\)''.

\begin{proposition}\label{factor-non-monic} Let \(f,g_0,h_0\in A[X]\) such that
\(\ov f=\ov{g_0} \ov{h_0}\) with \(\gcd(\ov{g_0},\ov{h_0})=1\) and
\(g_0\) is monic.    There exist \(g,h\in A[X]\) with
\(g\) monic, such that \(f=g h\), \(\ov g=\ov{g_0}\) and \(\ov
h=\ov{h_0}\). Moreover this factorization lifting is unique.
\end{proposition}

\begin{proof} Assume first that the discrete residual field has at
least \(1+\deg \ov f\) elements. We have some \(a\in A\) such that \(f(a)\in A^\times\),
so we can apply Lemma \ref{IfInfinite} to \(f(X+a)\).

In the general case we consider the subfield $\k_0$ generated
by the coefficients of $\ov{g_0}$ and $\ov{h_0}$. Since $\k$ is discrete
we are able either to find an element $a\in A$ such that    \(f(a)\in A^\times\)
or to assert that $\k_0$ is finite. In this case, we consider the subring
$A_0$ generated by the coefficients of $g_0$ and $h_0$, we localize this ring at 
the prime $\m \cap A_0$, and 
we consider the henselian subring $B_0\subseteq A$ it
generates. The morphism
$B_0\to A$ is local and the residue field of $B_0$ is $\k_0$.
We construct two finite extensions $\k_1$ and $\k_2$ of $\k_0$, each one 
containing an element which is not a root of $\ov f$. Moreover
$\k_1\cap\k_2=\k_0$. Let $p_i\in A_0[T_i]$ $(i=1,2)$ be monic polynomials
such that 
$\k_i=\k_0[T_i]\flr{\ov{p_i}(T_i)}$. Let $B_i=B_0[T_i]\flr{p_i(T_i)}$. 
By Corollary \ref{extension1}, $B_1$ and $B_2$ are Henselian.
By Lemma \ref{IfInfinite} we get a factorization $f(X)=g_i(X) h_i(X)$ inside each $B_i[X]$.
We have 
$$
B_i\subset B= B_0[T_1,T_2]\flr{p_1(T_1),p_2(T_2)}\simeq B_1\otimes_{B_0}B_2,
$$  which is a free $B_0$-module of rank $\deg(p_1) \deg(p_2)$. 
Inside $B[X]$ we get (by
unicity in Lemma \ref{IfInfinite}) $g_1=g_2$ and $h_1=h_2$,
and $g_i(X),h_i(X)\in B_1[X]\cap B_2[X]=B_0[X]\subset A[X]$.    
\end{proof}


\section{Henselization and strict Henselization of a local ring}

\begin{context} \label{context6}
In Section 6, \(A\) is a residually discrete
local ring with maximal ideal \(\m\) and residual field \(\k\).
\end{context}

\subsection{The Henselization}

 In this section we construct the Henselization of $A$ as a direct limit
 of extensions of $A$ that are obtained by adding inductively Hensel roots of monic  polynomials.  
 
 This kind of construction works for two reasons: the first one is that we are
 able to make a ``simple'' extension in a universal way. The second one is that 
 the universal property of simple extensions allows us to give canonical isomorphisms between two ``multiple'' extensions. In conclusion the system of multiple extensions
that we construct is an inductive system and does have a direct limit.

\def\A#1{A_#1}
\def\Af{\A f}

\subsubsection{One step}

\begin{defi} Let \(f(X)=X^n+\cdots+ a_1 X+a_0\in A[X]\) a monic
polynomial with \(a_1\in A^\times\) and \(a_0\in\m\). Then we denote
by \(\Af\) the ring defined as follows: if \(B=A[x]=A[X]\flr{f(X)}\)
(where \(x\) is the class of \(X\) in the quotient ring), let
\(S\subseteq B\) be the multiplicative part of \(B\) defined by
\[S=\{ g(x)\in B \tq g(X)\in A[X],\ g(0)\in A^\times \}.\]
Then by definition \(\Af\) is \(B\) localized in \(S\), that is
\(\Af=S^{-1} B\).
\end{defi}

We fix a polynomial \(f(X)\in A[X]\) such as in the above definition.

\begin{lemma} \label{DeAaAf}
The ring \(\Af\) is a residually discrete local ring.  Its maximal
ideal is \(\m \Af\).  Its residual field is (canonically
isomorphic to) \(\k\).  It is faithfully flat over \(A\). In
particular we can identify \(A\) with its image in \(\Af\), and write
\(A\subseteq\Af\).
\end{lemma}

\begin{proof} 
Since \(\Af\) is a localization of an algebra which is a free
\(A\)-module, \(\Af\) is flat over \(A\).  The elements of \(\Af\) can
be written formally as fractions \(r(x)/s(x)\) with \(r,s\in A[X]\),
\(s(0)\in A^\times\), \(r(x),s(x)\in B\).  Consider an arbitrary
\(a=r(x)/s(x)\in\Af\).  To prove that \(\Af\) is local and residually
discrete, we show that \(a\in\Af^\times\) or \(a\in\Rad(\Af)\).  If
\(r(0)\in A^\times\), then \(a\in\Af^\times\); if \(r(0)\in \m_{A}\), then
consider an arbitrary \(b=q(x)/s'(x)\in\Af\), we have
\(1+a b=(s(x) s'(x)+r(x) q(x))/(s(x)\cdot
s'(x))=p(x)/v(x)\) and \(p(0)\in A^\times\) so \(1+ab\in \Af^\times\),
and we are done.

We have shown that the morphism \(A\to\Af\) is local, so \(\Af\) is
faithfully flat over \(A\) (see lemma \ref{fideleplat}) and we
consider \(A\) as a subring of \(\Af\).

We have also shown that \(\m_{\Af}\) is the set of \(r(x)/s(x)\) with
\(r(0)\in\m\) (in particular \(\m\subseteq\m_{\Af}\)) and
\(\Af^\times\) is the set of \(r(x)/s(x)\) with \(r(0)\in A^\times\).
So in order to see that \(\m_{\Af}=\m \Af\) it is sufficient to
show that \(x/1\in \m \Af\).  Let
\[  y = x^{n-1} + a_{n-1} x^{n-2}+\cdots+ a_2 x + a_1 \]
We have \(y\in\Af^\times\), and \(y x=-a_0\), so \(x=
-a_0 y^{-1}\in\m \Af\).

An equality \(r(x)/s(x)=q(x)/u(x)\in \Af\) means an equality
\[v(X) (s(X) q(X)-u(X) r(X))\in \la f(X)\ra\]
in \(A[X]\) with \(v(0)\in A^{\times}\) and this implies that
\(s(0)q(0)-u(0)r(0)\in\m\). We deduce that the map \(\Af \ni
r(x)/s(x)\mapsto \ov {r(0)/s(0)}\in \k\) is a well defined ring
morphism. As its kernel is \(\m_{\Af}\) we obtain that the residual
field of \(\Af\) is canonically isomorphic to
\(\k\).
\end{proof}

In the following, as we did at the end of the proof, we denote \(x\)
for the element \(x/1\) of \(\Af\). It is a zero of \(f\) in
\(\m_{\Af}\).  But we note that \(A[x/1]\) as a subring of \(\Af\)
is a quotient of \(B=A[x]\).

\setlength{\unitlength}{1pt}
\def\hdarrow{{\begin{picture}(0,20)\put(2,16){\oval(4,4)[t]}
\put(0,16){\vector(0,-1){18}}
\end{picture}}}
\def\lnearrow_#1{{\begin{picture}(20,20)\put(0,0){\vector(3,2){20}}
\put(10,0){\(\sc #1\)}
\end{picture}}}
\def\larrow{\mathop{\begin{picture}(20,4)\put(0,2){\vector(1,0){20}}
\end{picture}}\limits}

\begin{lemma}[context \ref{context6}] \label{uniquembedding}
Let \(B,\m_B\) be a local ring and \(\phi:A\dans B\) a local
morphism. \\
Let \(f(X)=X^n+\cdots+a_1 X+a_0\in A[X]\) be a monic
polynomial with \(a_1\in A^\times\) and \(a_0\in\m\).
\\
If \(\phi(f)=X^n+\cdots+\phi(a_1) X+\phi(a_0)\in B[X]\) has a root
\(\xi\) in \(\m_B\), then there exists a unique local morphism
\(\psi:\Af\dans B\) such that \(\psi(x)=\xi\) and the following
diagram commutes:
\[\begin{matrix} A,\m & \larrow^\phi & B,\m_B\cr
        \hdarrow & \lnearrow_\psi \cr \Af,\m \Af
        \cr\end{matrix}\]
\end{lemma}

\begin{proof} \(\Af\) has been constructed exactly for this purpose.
\end{proof}

\subsubsection{An inductive definition}

\def\S{\mathcal{S}}

We now define an inductive system. Let \(\S\) be the smallest family
of local rings \((B,\m B)\) such that
\begin{enumerate}
\item \((A,\m)\in\S\);
\item if \((B,\m_B)\in\S\), \(f(X)=X^n+\cdots+a_1 X+a_0\in B[X]\) with \(a_1\in
B^\times\) and \(a_0\in\m_B\), then \((B_f,\m_{B_f})\) is in~\(\S\).
\end{enumerate}

Now we see that \(\S\) is an inductive system. The ring \(A\) is
canonically embedded in each local ring \((B,\m_B)\) in \(\S\), and
\(\m_B=\m B\). In a similar way, every local ring in \(\S\) is
canonically embedded in the ones which are constructed from it.

Given two elements \((B,\m_B)\) and \((C,\m_C)\) in \(\S\), the first
one is constructed by adding Hensel roots of successive polynomials
\(f_{1},\ldots ,f_{k}\) in successive extensions, the second one is
constructed by adding Hensel roots of successive polynomials
\(g_{1},\ldots ,g_{\ell}\) in successive extensions.  Now we can add
successively the Hensel roots of polynomials \(f_{1},\ldots ,f_{k}\)
to \(C\) and add successively the Hensel roots of polynomials
\(g_{1},\ldots ,g_{\ell}\) to \(B\).  It is easy to see that the
extension \(C'\) of \(C\) and the extension \(B'\) of \(B\) we have
constructed are canonically isomorphic.  
So we have a filtered inductive system all of whose morphisms are 
injective
and the inductive limit is a local ring that ``contains'' all
the elements of \(\S\) as subrings.

This kind of machinery always works when we have the property of
``unique embedding'' described in Lemma \ref{uniquembedding}. A similar
example is given by the construction of the real closure of an ordered
field (see \eg \cite{LR91}).

So we can define the {\em Henselization} of \(A\) by
\[A^h = {\lim_{\longrightarrow}}_{\>B\in\S} B. \]

We have the following theorem.

\begin{theorem}[context \ref{context6}] \label{henselization}
The ring \(A^h\) is a Henselian local ring with maximal ideal
\(\m A_h\). If \((B,\m_B)\) is a Henselian local ring and
\(\phi:A\dans B\) is a local morphism then there exists a 
unique local morphism \(\psi\) such that the following diagram
commutes:
\[\begin{matrix} A,\m & \larrow^\phi & B,\m_B\cr
        \hdarrow & \lnearrow_\psi \cr
            A^h,\m A^h \cr\end{matrix}\]
\end{theorem}

\begin{proof} Induction on the family \(\S\).
\end{proof}

\subsection{The strict Henselization}

\noindent
A ring is called a \emph{strict Henselian local ring} when it is 
local Henselian and the residue field is separably closed.

We want to construct a  
strict Henselian local ring associated to $A$ satisfying a universal property
similar to that given in Theorem \ref{henselization}.

We give only the sketch of the construction, which is very similar to the Henselization.

Moreover, we will assume that a separable closure of the residual field is given.

We have at the bottom the Henselization \(A^h\) of \(A\).
We need to construct a natural extension of \(A^h\)
having as residual field a separable closure of $\k$. 
  
\subsubsection{One step}

Using Corollary \ref{extension1} we can make some ``One step'' part of
the strict Henselization when we know an irreducible separable
polynomial  \(f(T)\) in \(\k[T]\). 
Consider the finite separable extension $\k[t]=\k[T]\flr{f(T)}$ of $\k$.

If \(F(U)\in A[U]\) gives \(f(U)\) modulo \(\m\) we consider the quotient
algebra \(A^{(F)}=A^h[u]=A^h[U]\flr{F(U)}\).  
By Corollary \ref{extension1}  we know that it is an
Henselian local ring with residue field
\(\k[t]\). More precisely it is a universal object of this kind,
as expressed in the following lemma.

\begin{lemma}\label{uniquestrict} Let $\varphi: A\to B$ be a local morphism
where $B$ is Henselian with residue field $\gl$. Assume that $f(T)$
has a root $t'$ in $\gl$ through the residual
map $\k\to\gl$. Then there exists a unique local morphism $A^{(F)}\to B$
which maps residually $t$ on $t'$. 
\end{lemma}

If \(F_{1}\in
A[V]\) gives also \(f(V)\) modulo \(\m\) let us call \(v\) the class of
\(V\) in \(A^{(F_1)}\). Lemma \ref{uniquestrict} 
implies that \(A^{(F)}=A^h[u]\)  and \(A^{(F_{1})}=A^h[v]\) are  canonically isomorphic: there is a root $u'$ of $F$ in  \(A^{(F_{1})}\) residually equal to $t$, and the isomorphism maps $u$ on $u'$.

In a similar way if \(x\in\k[z]\subseteq \k^{sep}\), \(x=p(z)\), and
\(G[T]\) is a polynomial giving modulo \(\m\) the minimal polynomial
of \(z\) we will have a canonical embedding of \(A^{(F)}\) in
\(A^{(G)}\) if we impose the condition that
\(P(\zeta)-\xi\in\m_{A^{(G)}}\) (here \(\zeta\) is the class of \(T\) 
in \(A^{(G)}\), and \(P\) is a polynomial giving \(p\) modulo \(\m\)).

\subsubsection{An inductive definition}

In order to have a construction of the strict Henselization 
as a usual ``static'' object we need a
\emph{separable closure} of \(\k\), \ie a discrete field
\(\k^{sep}\) containing \(\k\) with the following properties:
\begin{enumerate}
\item Every element \(x\in\k^{sep}\) is annihilated by an irreducible
separable polynomial in \(\k[T]\).
\item Every separable polynomial in \(\k[T]\) decomposes in linear
factors over~\(\k^{sep}\). 
\end{enumerate}

In that case we can define the strict Henselization through a new
inductive system, which is defined in a natural way from the
inductive system of finite subextensions of  \(\k^{sep}\).
We iterate the ``one step'' construction. The correctness of the glueing
of the corresponding extensions of \(A^h\) is obtained through 
Lemma~\ref{uniquestrict}.


\newpage

\small
\bibliographystyle{apalike}

\tableofcontents

\end{document}